\documentclass{article}
\usepackage{amsthm}
\usepackage{amsmath}
\usepackage{amsfonts}
\usepackage[letterpaper,body={13.5cm,21cm}, mag=1000]{geometry}
\usepackage{amssymb}
\usepackage{secdot}
\usepackage{setspace}
\usepackage{titlesec}
\titleformat{\section}[runin]{\bfseries\filcenter}{\thesection}{1em}{}

\renewcommand{\thesection}{\arabic{section}}
\title{\large \bf Isomorphism between automorphism groups of finitely generated groups}
\author{\small \bf Sandeep Singh\footnote{Supported by Council of Scientific and Industrial Research.} and Deepak Gumber\footnote{Supported by National Board for Higher Mathematics, Department of Atomic Energy.}\\
\small \em School of Mathematics and Computer Applications\\
\small \em Thapar University, Patiala - 147 004,
India\\
}
 
\date{}

\DeclareMathOperator{\Inn}{Inn}

\DeclareMathOperator{\Aut}{Aut}

\DeclareMathOperator{\Hom}{Hom}
\newtheorem{thm}{Theorem}[section]
\newtheorem{lm}[thm]{Lemma}

\newtheorem{cor}[thm]{Corollary}
\newtheorem{rem}[thm]{Remark}
\begin{document}

\maketitle
\begin{abstract}
\noindent {\bf Abstract.} Let $G$ be a finitely generated group and let $C^*$ denote the group of all central automorphisms of $G$ fixing the center of $G$ elementwise. Azhdari and Malayeri [J. Algebra Appl., {\bf 6}(2011), 1283-1290] gave necessary and sufficient conditions on $G$ such that $C^* \simeq \Inn(G)$. We prove a technical lemma and, as a consequence, obtain a short and easy proof of this result of Azhdari and Malayeri. Subsequently, we also obtain short proofs of some other existing and some new related results. 
\end{abstract}

\vspace{2ex}

\noindent {\bf 2010 Mathematics Subject Classification:}
20F28, 20F18.

\vspace{2ex}

\noindent {\bf Keywords:} Central-automorphism, nilpotent group.

\section{Introduction.}

Let $G$ be a finitely generated group and let $\Inn(G)$ denote the inner automorphism group of $G$. For normal subgroups $X$ and $Y$ of $G$, let $\Aut^X(G)$ and $\Aut_Y(G)$ denote the subgroups of $\Aut(G)$ centralizing $G/X$ and $Y$ respectively. We denote the intersection $\Aut^X(G)\cap\Aut_Y(G)$ by $\Aut_Y^X(G)$. Let $C^*$, in particular, denote the group $\Aut_{Z(G)}^{Z(G)}(G)$, where $Z(G)$ is the center of $G$. For a finite group $G$, let $G_p$ and $\pi(G)$ respectively denote the Sylow $p$-subgroup and the set of prime divisors of $G$. For a finite $p$-group $G$,  Attar \cite[Main Theorem]{att}  proved that $C^* = \Inn(G)$ if and only if either $G$ is abelian or $G$ is nilpotent of class 2 and $Z(G)$ is cyclic. Azhdari and Malayeri \cite[Theorem 0.1]{azh2011} (see also \cite[Theorem 2.3]{azh2013a} for correct version) generalized this result of Attar and proved that if $G$ is a finitely generated nilpotent group of class 2, then $C^*\simeq \Inn(G)$ if and only if $Z(G)$ is infinite cyclic or  $Z(G)\simeq C_{m}\times H\times \mathbb{Z}^{r}$, where $C_m\simeq \prod_{p\in\pi(G/Z(G))}Z(G)_p$, $H\simeq \prod_{p\notin\pi(G/Z(G))}Z(G)_p$, $r\geq 0$ is the torsion-free rank of $Z(G)$ and $G/Z(G)$ is of finite exponent dividing $m$. We prove a technical lemma, Lemma 2.1, and as a consequence give a short and easy proof of this main theorem of Azhdari and Malayeri. We also obtain short and alternate proofs of Corollary 2.1 of \cite{azh2013a},  and Propostion 1.11 and Theorem 2.2(i) of \cite{azh}.  Some other related results for finitely generated and finite $p$-groups are also obtained. 

  By $C_{p}$ we denote a cyclic group of order $p$ and by $X^n$ we denote the direct product of $n$-copies of a group $X$. By $\Hom(G,A)$ we denote the group of all homomorphisms of $G$ into an abelian group $A$. The rank of  $G$ is the smallest cardinality of a generating set of $G$. The torsion rank and torsion-free rank of $G$ are respectively denoted as $d(G)$ and $\rho(G)$. By  $\exp(G)$ we denote the exponent of torsion part of $G$. All other unexplained notations, if any, are standard. The  following well known results will be used very  frequently without further referring.

\begin{lm}
Let $U,V$ and $W$ be abelian groups. Then\\
$(i)$ if $U$ is torsion-free of rank $m$, then $\Hom(U, V)\simeq V^{m}$, and \\ 
$(ii)$ if $U$ is torsion and $V$ is torsion-free, then 
$\Hom(U, V)=1$.
\end{lm}

\section{Main results.}

Let $G$ be a finitely generated group and $M$ be an abelian subgroup of $G$ with $\pi(M)=\{q_1,q_2,\ldots ,q_e\}$. Let $L$ and $N$ be normal subgroups of $G$ such that $G^\prime\le N\le L$ and $\pi(G/L)=\pi(G/N)=\{p_1,p_2,\ldots ,p_d\}$. 
Let $X, Y, Z$  be respective torsion parts and $a,b,c$ be respective torsion-free ranks of $G/L,G/N$ and $M$. Let $X_{p_{i}}\simeq \prod_{j=1}^{l_{i}}C_{p_{i}^{\alpha_{ij}}}$, $Y_{p_{i}}\simeq \prod_{j=1}^{n_{i}}C_{p_{i}^{\beta_{ij}}} \;\; \mathrm{and} \;\;Z_{q_{i}}\simeq \prod_{j=1}^{m_{i}} C_{q_{i}^{\gamma_{ij}}}$, where for each $i$, $ \alpha_{ij}\geq \alpha_{i(j+1)},\; \beta_{ij}\geq \beta_{i(j+1)}$ and $ \gamma_{ij}\geq \gamma_{i(j+1)}$ are positive integers, respectively denote the Sylow subgroups of $X$, $Y$ and $Z$. Then 
$$G/L\simeq X\times \mathbb{Z}^a\simeq\prod_{i=1}^d X_{p_i}\times\mathbb{Z}^a\simeq\prod_{i=1}^d\;\;
\prod_{j=1}^{l_{i}}C_{p_{i}^{\alpha_{ij}}}\times \mathbb{Z}^a,$$ 
$$G/N\simeq Y\times \mathbb{Z}^b\simeq \prod_{i=1}^d Y_{p_i}\times \mathbb{Z}^b\simeq\prod_{i=1}^d\;\;
\prod_{j=1}^{n_{i}}C_{p_{i}^{\beta_{ij}}}\times \mathbb{Z}^b$$ and 
$$M\simeq Z\times \mathbb{Z}^c\simeq \prod_{i=1}^e Z_{q_i}\times \mathbb{Z}^c\simeq\prod_{i=1}^e\;\;\prod_{j=1}^{m_{i}} C_{q_{i}^{\gamma_{ij}}}\times\mathbb{Z}^c .$$
Since $G/L$ is a quotient group of $G/N$, it follows that  $a\le b,\, l_i\le n_i$ and $\alpha_{ij}\le\beta_{ij}$ for all $i,1\le i\le d$ and for all $j,1\le j\le l_i$. 
We begin with the following lemma.

\begin{lm}  Let $G,L,M$ and $N$ be as above. Then $\Hom(G/N, M)\simeq G/L$ if and only if one of the following conditions hold:
\begin{enumerate}
\item[$(i)$] $G$ is torsion-free, $M$ is infinite cyclic and both $G/L$ and $G/N$ are torsion-free of same rank.
\item[$(ii)$] $G$ is torsion, $M\simeq C_{\prod_{i=1}^{d}{p_{i}^{\gamma_{i1}}}}\times\prod_{i=d+1}^e Z_{q_i}$, $l_{i}=n_{i}$ and either $\alpha_{ij}=\beta_{ij}\le\gamma_{i1}$ for each $j$ or $\alpha_{ij}=\gamma_{i1}$ for $ 1\le j\leq r_{i}$ and $\alpha_{ij}=\beta_{ij}$ for $r_{i}+1\leq j\leq l_{i}$, where $r_{i}$ is the largest positive integer between $1$ and $ l_{i}$ such that $\beta_{ir_{i}}>\gamma_{i1}$ for each fixed $i, 1\le i \le d$.
\item[$(iii)$] $G$ is a mixed group,  $M\simeq C_{\prod_{i=1}^{d}{p_{i}^{\gamma_{i1}}}}\times\prod_{i=d+1}^e Z_{q_i} \times \mathbb{Z}^c$, both $G/L$ and $G/N$ are finite, $l_{i}=n_{i}$ and either $\alpha_{ij}=\beta_{ij}\le\gamma_{i1}$ for each $j$ or $\alpha_{ij}=\gamma_{i1}$ for $ 1\le j\leq r_{i}$ and $\alpha_{ij}=\beta_{ij}$ for $r_{i}+1\leq j\leq l_{i}$, where $r_{i}$ is the largest positive integer between $1$ and $ l_{i}$ such that $\beta_{ir_{i}}>\gamma_{i1}$ for each fixed $i, 1\le i \le d$.
\end{enumerate}
\end{lm}
\begin{proof} It is easy to see that if any of the three conditions hold, then $\Hom(G/N, M)\simeq G/L$. Conversely suppose that $\Hom(G/N, M)\simeq G/L$. Then
\begin{equation}
\Hom(Y\times \mathbb{Z}^b, Z\times \mathbb{Z}^c)\simeq X\times \mathbb{Z}^a.
\end{equation}
We prove only (i) and (ii), because (iii) can be proved using similar arguments.
First assume that $G$ is torsion-free. Then $N$ is also torsion-free and therefore by (1) $\Hom(Y\times \mathbb{Z}^b, \mathbb{Z}^c)\simeq X\times \mathbb{Z}^a$. Thus $X=1$ and since $a\le b$, $c=1$ and $a=b$. It follows that $M$ is infinite cyclic and both $G/N$ and $G/L$ are torsion-free of same rank. Next assume that $G$ is torsion. Then $\Hom(Y,Z)\simeq X$ by (1). Since $\pi(X)=\pi(Y)$ and $d(X)\le d(Y)$, therefore $q_i=p_i$ and $m_i=1$ for all $i, 1\le i\le d$. Thus $M\simeq \prod_{i=1}^{d}C_{p_{i}^{\gamma_{i1}}}\times
\prod_{i=d+1}^e\prod_{j=1}^{m_{i}}C_{q_{i}^{\gamma_{ij}}}$.
 Also, observe that \[\begin{array}{lcl}
\Hom(Y,Z)&\simeq&\Hom(\displaystyle\prod_{i=1}^d\displaystyle\prod_{j=1}^{n_{i}}
C_{p_{i}^{\beta_{ij}}}, \displaystyle \displaystyle{\prod_{i=1}^{d}C_{p_{i}^{\gamma_{i1}}}}\times\displaystyle
\prod_{i=d+1}^e\prod_{j=1}^{m_{i}}C_{q_{i}^{\gamma_{ij}}})\\
&\simeq&\Hom(\displaystyle\prod_{i=1}^{d}\displaystyle\prod_{j=1}^{n_{i}}
C_{p_{i}^{\beta_{ij}}}, \displaystyle\prod_{i=1}^{d}C_{p_{i}^{\gamma_{i1}}})\\
&\simeq&\displaystyle\prod_{i=1}^d\Hom(\displaystyle\prod_{j=1}^{n_{i}}
C_{p_{i}^{\beta_{ij}}}, C_{p_{i}^{\gamma_{i1}}})
\end{array}\]
and $X\simeq\prod_{i=1}^d\prod_{j=1}^{l_{i}}
C_{p_{i}^{\alpha_{ij}}}$. Therefore
$\Hom(\prod_{j=1}^{n_{i}}
C_{p_{i}^{\beta_{ij}}}, C_{p_{i}^{\gamma_{i1}}})\simeq \prod_{j=1}^{l_{i}}C_{p_{i}^{\alpha_{ij}}}$
for each $i,1\le i\le d$, and hence $l_i=n_i$. It thus follows that for each fixed $i, 1\le i\le d$,
\begin{equation}
\Hom(\displaystyle\prod_{j=1}^{l_{i}}
C_{p_{i}^{\beta_{ij}}}, C_{p_{i}^{\gamma_{i1}}})\simeq \displaystyle\prod_{j=1}^{l_{i}}
C_{p_{i}^{\alpha_{ij}}}.
\end{equation}
Now, if $\exp(Y_{p_{i}})\le \exp(Z_{p_{i}})$, then $\beta_{ij}\le\gamma_{i1}$ for each $j$ and  
$\Hom(\prod_{j=1}^{l_{i}}
C_{p_{i}^{\beta_{ij}}}, C_{p_{i}^{\gamma_{i1}}})\simeq \prod_{j=1}^{l_{i}}C_{p_{i}^{\beta_{ij}}}$. It therefore follows from (2) that $\alpha_{ij}=\beta_{ij}$ for each $j$. And, if $\exp(Y_{p_{i}})>\exp(Z_{p_{i}})$, then there exists largest positive integer $r_{i}$ between $1$ and $l_i$ such that $ \beta_{ir_{i}}>\gamma_{i1}$ and $\beta_{ij}\le \gamma_{i1}$ for each $j,r_i+1\le j\le l_i$. Therefore 
$\Hom(\prod_{j=1}^{l_{i}}
C_{p_{i}^{\beta_{ij}}}, C_{p_{i}^{\gamma_{i1}}})\simeq \prod_{j=1}^{r_{i}}C_{p_{i}^{\gamma_{i1}}}\times \prod_{j=r_i+1}^{l_{i}}
C_{p_{i}^{\beta_{ij}}}$. It then follows by (2) that $\alpha_{ij}=\gamma_{i1}$ for $1\le j\le r_i$ and $\alpha_{ij}=\beta_{ij}$ for $r_i+1\le j\le l_i$.
\end{proof}

\begin{rem} {\em  Observe that if $N=L$ and $\exp(G/N)|\exp(M)$, then  $\exp(Y_{p_{i}})\le \exp(Z_{p_{i}})$ for all $i$ and hence $\Hom(G/L, M)\simeq G/L$ if and only if either $M$ is infinite cyclic or  $M\simeq C_{\prod_{i=1}^{d}{p_{i}^{\gamma_{i1}}}}\times\prod_{i=d+1}^e Z_{q_i} \times \mathbb{Z}^c$, where  $c\geq 0$ is the torsion-free rank of $M$.}
\end{rem}

The next lemma is a little modification of arguments of Alperin \cite[Lemma 3]{alp} and Fournelle \cite[Section 2]{fou}.

\begin{lm} \label{Lemma1}
Let $G$ be any group and $Y$ be a central subgroup of $G$ contained in a normal subgroup $X$ of $G$. Then the group of all automorphisms of $G$ that induce the identity on both $X$ and $G/Y$ is isomorphic to
$\Hom(G/X,Y)$. 
\end{lm}

Observe  that $C^*\simeq\Hom(G/Z(G), Z(G))$ by Lemma 2.3. If $G$ is nilpotent of class 2, then $\exp(G')=\exp(G/Z(G))$. Now taking $L=M=N=Z(G)$ in Lemma 2.1, we get the following main result of Azhdari and Malayeri \cite[Theorem 0.1]{azh2011} (see \cite[Theorem 2.3]{azh2013a} for correct version).

\begin{cor}
Let $G$ be a finitely generated nilpotent group of class $2$. Then  $C^*\simeq\Inn(G)$ if and only if either $Z(G)$ is infinite cyclic or $Z(G)\simeq C_{\prod_{i=1}^{d}{p_{i}^{\gamma_{i1}}}}\times\prod_{i=d+1}^e Z_{q_i} \times \mathbb{Z}^c$, where $c$ is the torsion-free rank of $Z(G)$.
\end{cor}

\begin{cor} [{\cite[Corollary 2.1]{azh2013a}}]
Let $G$ be a finitely generated non-abelian group and let $M$ and $N$ be normal subgroups of $G$ such that $M\le Z(G)\le N$ and $G/Z(G)$ is finite. Then $\mathrm{Aut}^M_{N}(G)=\Inn(G)$ if and only if $G$ is a nilpotent group of class $2$, $N=Z(G)$, $G^\prime\leq M$ and $M\simeq C_{\prod_{i=1}^{d}{p_{i}^{\gamma_{i1}}}}\times\prod_{i=d+1}^e Z_{q_i} \times \mathbb{Z}^c$, where  $c\geq 0$ is the torsion-free rank of $M$.
\end{cor}
\begin{proof} First suppose that  $\mathrm{Aut}^M_{N}(G)=\Inn(G)$. Observe that $\mathrm{Aut}^M_{N}(G)\simeq \Hom(G/N, M)$ by Lemma 2.3. It follows that $\Inn(G)$ is abelian and therefore nilpotence class of $G$ is $2$. For any $[a, b]\in G^\prime$,  $[a,b]=a^{-1}I_b(a)\in M$ and thus $G^\prime\le M$. Also, for any $n\in N$, $I_x(n)=n$ for all $x\in G$ and therefore $N=Z(G)$. Now since $\exp(G/Z(G))=\exp(G')$ divides $\exp(M)$, the result follows from Lemma 2.1 by taking $L=Z(G)$. The converse follows easily.
\end{proof}

In 1911, Burnside \cite[Note B, p. 463]{bur} gave the notion of pointwise inner automorphism of a group $G$. An automorphism $\alpha$ of $G$ is called pointwise inner automorphism of $G$ if $x$ and $\alpha(x)$ are conjugate for each $x\in G$. Let $H$ be a characteristic subgroup of $G$. As defined in \cite{azh}, an automorphism $\alpha$ of $G$ is called $H$-pointwise inner if for each element $x\in G$, there exists $h\in H$ such that $\alpha(x) = x^h = x[x, h].$ For convenience, we denote $\gamma_{k}(G)$-pointwise inner automorphism of $G$ by $\mathrm{Aut}_{k-pwi}(G)$. As another application of Lemma 2.1, we get the following two results of Azhdari \cite{azh}.   
The second one generalizes  Theorem 2.2(i) of \cite{azh}.

\begin{cor} [{\cite[Prop. 1.11]{azh}}] Let $G$ be a finitely generated nilpotent group of class $k+1\geq 2$. Then $\Hom(G/\zeta_{k}(G),\gamma_{k+1}(G))\simeq G/\zeta_{k}(G)$ if and only if  $\gamma_{k+1}(G)$ is  cyclic. In particular, if $\gamma_{k+1}(G) = [x, \gamma_{k}(G)]$ for all $x\in G\setminus C_{G}(\gamma_{k}(G))$ is cyclic, then $\mathrm{Aut}_{k-pwi}(G)$ is isomorphic to a quotient group of $\Inn(G)$.
\end{cor}

\begin{proof} It follows from \cite[Cor. 2.6, Cor. 3.16, Cor. 3.17]{war} that $\exp(G/\zeta_{k}(G))=\exp(\gamma_{k+1}(G))$ and $G/\zeta_{k}(G)$ is finite if and only if $\gamma_{k+1}(G)$ finite. The result now follows from Lemma 2.1 (see Remark 2.2) by taking $L=N=\zeta_{k}(G)$ and $M=\gamma_{k+1}(G)$.
In particular, if $\gamma_{k+1}(G) = [x, \gamma_{k}(G)]$ for all $x\in G\setminus C_{G}(\gamma_{k}(G))$ is cyclic, then using the arguments as in \cite[Prop. 3.1]{yad}, we can prove that $\mathrm{Aut}_{k-pwi}(G)\simeq\Hom(G/\zeta_{k}(G),\gamma_{k+1}(G))$.
\end{proof}

\begin{cor} [{\em cf.} {\cite[Theorem 2.2(i)]{azh}}] Let $G$ be a finitely generated nilpotent group of class $k+1\geq 2$. Then $\Hom(G/\zeta_{k}(G),\gamma_{k+1}(G))\simeq \Inn(G)$ if and only if $G$ is nilpotent of class $2$ and $G^\prime$ is cyclic. In particular, if $\gamma_{k+1}(G) = [x, \gamma_{k}(G)]$ for all $x\in G\setminus C_{G}(\gamma_{k}(G))$, then $\Aut_{k-pwi}(G)\simeq\Inn(G)$ if and only if $G$ is nilpotent of class $2$ and $G^\prime$ is cyclic.
\end{cor}

\begin{proof} Observe that if $\Hom(G/\zeta_{k}(G),\gamma_{k+1}(G))\simeq \Inn(G)$, then $G/Z(G)$ is abelian, and therefore nilpotence class of $G$ is $2$. It follows that $\zeta_{k}(G)=Z(G)$ and $\gamma_{k+1}(G)=G^\prime$. The result now follows from above corollary by taking $k=1$.
\end{proof}

For $g\in G$ and $\alpha\in\Aut(G)$, the element $[g,\alpha]=g^{-1}\alpha(g)$ is called the autocommutator of $g$ and $\alpha$. Inductively, define 
$$[g,\alpha_1,\alpha_2,\ldots ,\alpha_n]=[[g,\alpha_1,\alpha_2,\ldots ,\alpha_{n-1}],\alpha_n],$$
where $\alpha_i\in\Aut(G)$. The absolute center $L(G)$ of $G$ is defined as
$$L(G)=\{g\in G\,|\,[g,\alpha]=1, \;\mbox{for all}\; \alpha\in \Aut(G)\}.$$ 
Let $L_1(G)=L(G)$, and for $n\ge 2$, define $L_n(G)$ inductively as 
$$L_n(G)=\{g\in G\,|\, [g,\alpha_1,\alpha_2,\ldots ,\alpha_n]=1 \;\mbox{for all}\; \alpha_1,\alpha_2,\ldots ,\alpha_n\in\Aut(G)\}.$$ 
The autocommutator subgroup $G^*$ of $G$ is defined as
$$G^*=\langle g^{-1}\alpha(g)\,|\,g\in G,\alpha\in \Aut(G)\rangle.$$
It is easy to see that  
$L_n(G)\le Z_n(G)$ for all $n\ge 1$ and $G'\le G^*$.  An automorphism $\alpha$ of $G$ is called an autocentral automorphism if $g^{-1}\alpha(g)\in L(G)$ for all $g\in G$. The group of all autocentral automorphisms of $G$ is denoted by $\mathrm{Var}(G).$ A group $G$ is called autonilpotent of class at most $n$ if $L_{n}(G) = G$ for some natural number $n$. Observe that if $G$ is autonilpotent of class 2, then $G^*\le L(G)$. Nasrabadi and Farimani \cite{nas} proved that if $G$ is a finie autonilpotent $p$-group of class 2, then $\mathrm{Var}(G)=\Inn(G)$ if and only if $L(G)=Z(G)$ and $Z(G)$ is cyclic. Observe that $\mathrm{Var}(G)\simeq \Hom(G/L(G), L(G))$ by Lemma 2.3. As a final consequence of Lemma 2.1, we get the following result which generalizes  the main result of Nasrabadi and Farimani. The proof follows from Lemma 2.1 by taking $M=N=L(G)$ and $L=Z(G)$.

\begin{cor}  Let $G$ be a finitely generated non-abelian group such that  $G^{\prime}\le L(G)$ and  $\pi(G/L(G))=\pi(G/Z(G))$. Then $\mathrm{Var}(G)\simeq \Inn(G)$ if and only if one of the following holds 
\begin{enumerate}
\item[$(i)$] $G$ is torsion-free, $L(G)$ is infinite cyclic and $\rho(G/L(G))=\rho(G/Z(G))$;
\item[$(ii)$] $G$ is torsion, $L(G)\simeq C_{\prod_{i=1}^{d}{p_{i}^{\gamma_{i1}}}}\times\prod_{i=d+1}^e Z_{q_i}$ and either $L(G)=Z(G)$ or $l_{i}=n_{i}$, $\alpha_{ij}=\gamma_{i1}$ for $ 1\le j\leq r_{i}$ and $\alpha_{ij}=\beta_{ij}$ for $r_{i}+1\leq j\leq l_{i}$, where $r_{i}$ is the largest positive integer between $1$ and $ l_{i}$ such that $\beta_{ir_{i}}>\gamma_{i1}$ for each fixed $i, 1\le i \le d$.
\item[$(iii)$] $G$ is a mixed group, both $G/L(G)$ and $G/Z(G)$ are finite, $L(G)\simeq C_{\prod_{i=1}^{d}{p_{i}^{\gamma_{i1}}}}\times\prod_{i=d+1}^eZ_{q_i} \times \mathbb{Z}^c$ and either $L(G)=Z(G)$ or $l_{i}=n_{i}$, $\alpha_{ij}=\gamma_{i1}$ for $ 1\le j\leq r_{i}$ and $\alpha_{ij}=\beta_{ij}$ for $r_{i}+1\leq j\leq l_{i}$, where $r_{i}$ is the largest positive integer between $1$ and $ l_{i}$ such that $\beta_{ir_{i}}>\gamma_{i1}$ for each fixed $i, 1\le i \le d$.
\end{enumerate}
\end{cor}

Let $G$ be a finite $p$-group such that $G^\prime \le L(G)$. Let
$G/Z(G)\simeq\prod_{i=1}^r C_{p^{\alpha_i}}$,
$G/L(G)\simeq\prod_{i=1}^s C_{p^{\beta_j}}$ and $L(G)\simeq \prod_{i=1}^t C_{p^{\gamma_i}}$, where  $\alpha_1\geq\alpha_2\geq\ldots\geq\alpha_r$, $\beta_1\geq\beta_2\geq\ldots\geq\beta_s$ and  $\gamma_1\geq \gamma_2\geq\ldots\geq \gamma_t $ are positive integers. Since $G/Z(G)$ is a quotient group of $G/L(G)$, $r\le s$ and $\alpha_i\le \beta_i$ for $1\le i\le r$.

 \begin{cor}
Let $G$ be a finite non-abelian p-group. Then $\mathrm{Var}(G)=\Inn(G)$ if and only if $G^{\prime}\le L(G) $,  $L(G)$ is cyclic and either  $L(G)=Z(G)$ or $d(G/L(G))=d(G/Z(G))$, $\alpha_i=\gamma_1$ for $1\le i\le k$ and $\alpha_i =\beta_i$ for $k+1\le i\le r$, where $k$ is the largest positive integer such that $\beta_k>\gamma_1.$ 
\end{cor}
 
\begin{proof} Observe that if $\mathrm{Var}(G)=\Inn(G)$, then for any $[a, b]\in G^\prime$, $[a,b]=a^{-1}I_b(a)\in L(G)$ and thus $G^\prime\le L(G)$. The result now follows from Cor. 2.8.
\end{proof}

\begin{cor} [{\cite[Theorem 3.2]{nas}}] Let $G$ be a non-abelian autonilpotent finite $p$-group of class $2$. Then $\mathrm{Var}(G)=\Inn(G)$ if and only if $L(G)=Z(G)$ and $L(G)$ is cyclic.
\end{cor}
\begin{proof}
Suppose that $\mathrm{Var}(G)=\Inn(G)$. Observe that if $g^{-1}\alpha(g)\in G^*$, then $\alpha(g)=gl$ for some $l\in L(G)$ and hence $(g^{-1}\alpha(g))^{m}=g^{-m}\alpha(g)^{m}$ for all $m\ge 1$. Let $\exp(G/L(G))=d$ and $\exp(G^*)=k$. Then $1=(g^{-1}\alpha(g))^k=g^{-k}\alpha(g)^{k}$ implies that $g^k\in L(G)$ and hence $d\le k$. Conversely, if $gL(G)\in G/L(G)$, then $g^d\in L(G)$ and thus $1=g^{-d}\alpha(g^d)=(g^{-1}\alpha(g))^d$. It follows that $k\le d$ and hence $\exp(G/L(G))=\exp(G^*)$. Since $G^*\le L(G)$, $\exp(G/L(G))|\exp(L(G))$. 
Therefore $\mathrm{Var}(G)\simeq\Hom(G/L(G),L(G))\simeq G/L(G)$, because $L(G)$ is cyclic by Corollary 2.9, and hence $L(G)=Z(G)$.
\end{proof}

\end{document}